\documentclass[12pt]{amsart}  

\usepackage[latin1]{inputenc}
\usepackage{amsmath} 
\usepackage{amsfonts}
\usepackage{amssymb}
\usepackage{stmaryrd}
\usepackage{latexsym} 
\usepackage{graphicx}
\usepackage{subfigure}
\usepackage{color}
\usepackage{hyperref}
\usepackage{verbatim}
\usepackage[all]{xy}
\usepackage{graphics}
\usepackage{dutchcal}
\usepackage{shuffle}
\usepackage{mathrsfs}
\usepackage{pdfsync}
\usepackage{tikz}
\usetikzlibrary{cd}
\usepackage{url}
\usepackage{enumerate}

\newcommand{\tcr}[1]{\textcolor{black}{#1}}
\newcommand{\VW}[1]{\textcolor{black}{#1}}
\newcommand{\FA}[1]{\textcolor{black}{#1}}

\newcommand{\svw}[1]{\textcolor{black}{#1}}

\theoremstyle{definition}
\newtheorem{theorem}{Theorem}[section]
\newtheorem*{theorem*}{Theorem}
\newtheorem{proposition}[theorem]{Proposition}

\newtheorem{lemma}[theorem]{Lemma}

\newtheorem{definition}[theorem]{Definition}
\newtheorem{example}[theorem]{Example}

\theoremstyle{remark}
\newtheorem{remark}[theorem]{Remark}
\DeclareMathOperator{\dotcup}{{\dot\cup}}

\numberwithin{equation}{section}
\newcommand{\Sym}{\ensuremath{\operatorname{Sym}}}

\newcommand{\QSym}{\ensuremath{\operatorname{QSym}}}
\newcommand{\p}{\mathcal p} 

\newcommand{\NSym}{\ensuremath{\operatorname{NSym}}}

\newcommand{\set}{\operatorname{set}}
\newcommand{\comp}{\operatorname{comp}}

\newcommand{\ind}{\operatorname{ind}}

\newlength\cellsize 
\savebox2{%
\begin{picture}(15,15)
\put(0,0){\line(1,0){15}}
\put(0,0){\line(0,1){15}}
\put(15,0){\line(0,1){15}}
\put(0,15){\line(1,0){15}}
\end{picture}}
\newcommand\cellify[1]{\def\thearg{#1}\def\nothing{}%
\ifx\thearg\nothing
\vrule width0pt height\cellsize depth0pt\else
\hbox to 0pt{\usebox2\hss}\fi%
\vbox to 15\unitlength{
\vss
\hbox to 15\unitlength{\hss$#1$\hss}
\vss}}
\newcommand\tableau[1]{\vtop{\let\\=\cr
\setlength\baselineskip{-16000pt}
\setlength\lineskiplimit{16000pt}
\setlength\lineskip{0pt}
\halign{&\cellify{##}\cr#1\crcr}}}
\savebox3{%
\begin{picture}(15,15)
\put(0,0){\line(1,0){15}}
\put(0,0){\line(0,1){15}}
\put(15,0){\line(0,1){15}}
\put(0,15){\line(1,0){15}}
\end{picture}}
\newcommand\expath[1]{%
\hbox to 0pt{\usebox3\hss}%
\vbox to 15\unitlength{
\vss
\hbox to 15\unitlength{\hss$#1$\hss}
\vss}}
\newcommand\bas[1]{\omit \vbox to \cellsize{ \vss \hbox to \cellsize{\hss$#1$\hss} \vss}}

\begin{document}

\title[$P$-partition power sums]{$P$-partition power sums}

\author{Farid Aliniaeifard}
\address{
 Department of Mathematics,
 University of British Columbia,
 Vancouver BC V6T 1Z2, Canada}
\email{farid@math.ubc.ca}

\author{Victor Wang}
\address{
 Department of Mathematics,
 University of British Columbia,
 Vancouver BC V6T 1Z2, Canada}
\email{vyzwang@student.ubc.ca}

\author{Stephanie van Willigenburg}
\address{
 Department of Mathematics,
 University of British Columbia,
 Vancouver BC V6T 1Z2, Canada}
\email{steph@math.ubc.ca}

\thanks{All authors were supported  in part by the National Sciences and Engineering Research Council of Canada.}

\begin{abstract} 
\tcr{We develop the theory of weighted $P$-partitions, which generalises the theory of $P$-partitions from labelled posets to weighted labelled posets. We define the related generating functions in the natural way and compute their product, coproduct and other properties. As an application we introduce
 the basis of combinatorial power sums for the Hopf algebra of quasisymmetric functions and the reverse basis, both of which refine the power sum symmetric functions. These bases} share many properties with the type 1 and type 2 quasisymmetric power sums introduced by Ballantine, Daugherty, Hicks, Mason and Niese, and moreover expand into the monomial basis of quasisymmetric functions with nonnegative integer coefficients. We prove formulas for products, coproducts and classical quasisymmetric involutions via the combinatorics of $P$-partitions, and give combinatorial interpretations for the coefficients when expanded into the monomial and fundamental bases.
\end{abstract}

\maketitle
\tableofcontents

\section{Introduction}\label{sec:intro}
The theory of quasisymmetric functions was first introduced in the thesis of Stanley \cite{StanP}, before being developed in the seminal paper of Gessel \cite{Gessel} using the theory of $P$-partitions. Quasisymmetric functions play an important role in algebraic combinatorics, having connections to other areas such as discrete geometry through posets \cite{Ehrenborg, Gessel, Stembridge}, and to probability theory through riffle shuffles \cite{Stanshuf}. They also form a Hopf algebra $\QSym$ \cite{Ehrenborg, Gessel, Malvenuto-Reutenauer}, which is the terminal object in the category of combinatorial Hopf algebras \cite{ABS} and generalises the Hopf algebra of symmetric functions $\Sym$ \cite{Gessel}. This connection to symmetric functions is often used to gain new insight into questions about symmetric functions, for example \cite{AlexSulz, BTvW, SW}.

Quasisymmetric analogues of classical symmetric function bases facilitate the translation between the language of symmetric functions and the language of quasisymmetric functions, and have been discovered for monomial symmetric functions \cite{Gessel}, Schur functions \cite{BBSSZ, HLMvW} and, most recently, for power sum symmetric functions \cite{BDHMN}. Power sum symmetric functions  are of especial interest because they encode the class values of the characters of the symmetric group under the Frobenius characteristic map, for example \cite{EC2}. 
In this paper we introduce a new quasisymmetric power sum basis defined combinatorially in terms of $P$-partitions, returning to the foundations of $\QSym$.

We begin by giving the necessary background in Section~\ref{sec:bg}. In Section~\ref{sec:P-part} we introduce the quasisymmetric generating function $K_{(P,\gamma,w)}$ of $P$-partitions associated \FA{with} a weighted labelled poset $(P,\gamma,w)$ in Definition~\ref{def:KPgw} and study its basic properties. Next, in Section~\ref{sec:qsymnsym} we motivate properties of a quasisymmetric power sum basis, and investigate the relationship between power sum quasisymmetric functions and their scaled duals in Proposition~\ref{prop:dual}. In Section~\ref{sec:power}, we define the bases of combinatorial power sums $\{\p_\alpha\}_{\alpha\vDash n\ge 0}$ and reverse combinatorial power sums $\{\p^r_\alpha\}_{\alpha\vDash n\ge 0}$ in Definitions~\ref{def:cpow} and \ref{def:rcpow}. We prove, using the combinatorics of $P$-partitions developed in Section~\ref{sec:P-part}, rules for products, coproducts and the quasisymmetric involutions $\psi$, $\rho$ and $\omega$ in Theorems~\ref{the:pprod}, \ref{the:pcop} and \ref{the:pinv}. Finally, in Theorems~\ref{the:ptomqsym} and \ref{the:ptofqsym}, we give combinatorial interpretations in terms of enumerating certain matrices for the coefficients of $\p_\alpha$ in the monomial and fundamental bases.

\section{Preliminaries}\label{sec:bg}
{A \textit{partially ordered set}, or simply \textit{poset}, is a set $P$ equipped with a binary relation $\le$ that is reflexive, antisymmetric and transitive. We also write $p< q$ for $p,q\in P$ if and only if $p\le q$ and $p\neq q$.} {A \textit{lower set} of a poset $P$ is a subset $I\subseteq P$ such that $x\in I$ and $y\le x$ implies $y\in I$.} {A poset in which every pair of elements $p,q$ satisfies $p\le q$ or $p\ge q$ is a \textit{chain}. A \textit{linear extension} $s$ of a poset $P$ is a chain on the same underlying set of $P$ such that $p\le q$ in $s$ whenever $p\le q$ in $P$. The \textit{dual poset} $P^*$ of a poset $P$ is the poset with the same underlying set as $P$ such that $p\le q$ in $P^*$ if and only if $p\ge q$ in $P$.} Given two posets $P,Q$, we write $P\dotcup Q$ to denote the poset on their disjoint union with comparisons inherited from $P$ and $Q$. 

We next introduce the theory of $P$-partitions, \tcr{\cite{StanP}, although we use order-preserving, rather than order-reversing, maps.}
\begin{definition}
Let $P$ be a finite poset. A \textit{labelling} of $P$ is an injective map $\gamma$ from $P$ to a chain. We call the pair $(P,\gamma)$ a \textit{labelled poset}.
\end{definition}

If $\gamma:P\to s$ is a labelling of $P$, where $P$ is a finite poset and $s$ is a chain, the \textit{dual labelling} is $\gamma^*:P\to s^*$, defined by $\gamma^*(p)=\gamma(p)$ for all $p\in P$.

\begin{definition}\label{def:P-part}
Let $(P,\gamma)$ be a labelled poset. A \tcr{\emph{$(P,\gamma)$-partition}} is a map \tcr{$f:P\to\mathbb{N} = \{1,2, \ldots\}$} satisfying, for all $p<q$ in $P$,
\begin{enumerate}
    \item $f(p)\le f(q)$, that is, $f$ is order-preserving,
    \item $f(p)=f(q)$ implies $\gamma(p)<\gamma(q)$.
\end{enumerate}
We denote by $\mathscr O (P,\gamma)$ the set of all $(P,\gamma)$-partitions.
\end{definition}

The following lemma describes the set of $P$-partitions associated \FA{with} a labelled poset in terms of those of its linear extensions.

\begin{lemma}[Fundamental lemma of $P$-partitions]\tcr{{\cite[Theorem 6.2]{StanP}}}\label{lem:P-part}
Let $(P,\gamma)$ \tcr{be} a labelled poset. Then
$$\mathscr O(P,\gamma)=\dot\bigcup \,\mathscr O(s,\gamma),$$
where the disjoint union is over all linear extensions $s$ of $P$.
\end{lemma}

\begin{example}\label{ex:ppartition}
{Consider the poset $P=\{p,q\}$, where $p$ and $q$ are incomparable. Let $\gamma:P\to\mathbb N$ be the labelling such that $\gamma(p)=1$ and $\gamma(\tcr{q})=2$. Let $s_1$ and $s_2$ be the linear extensions of $P$ given by $p<q$ and $q<p$, respectively.}

{Then the fundamental lemma of $P$-partitions says that $\mathscr O(P,\gamma) = \mathscr O(s_1,\gamma)\dotcup \mathscr O(s_2,\gamma).$ Since $p$ and $q$ are incomparable in $P$, the set of $(P,\gamma)$-partitions consists of all maps $f:P\to \mathbb N$. Note that given such a map $f:P\to\mathbb N$, either $f(p)\le f(q)$, in which case $f$ is an $(s_1,\gamma)$-partition, or $f(q)<f(p)$, in which case $f$ is an $(s_2,\gamma)$-partition.}
\end{example}

A \textit{composition} $\alpha=\alpha_1\cdots\alpha_{\ell(\alpha)}$ is a finite ordered list of positive integers, where $\ell(\alpha)$ is the \textit{length} of $\alpha$. When $\alpha_{j+1}=\dots=\alpha_{j+m}=i$, we often abbreviate this sublist to $i^m$. We call the integers {$\alpha_i$} the \textit{parts} of $\alpha$. The \textit{size} of $\alpha$ is $|\alpha|=\alpha_1+\dots+\alpha_{\ell(\alpha)}$. If $|\alpha|=n$, we say that $\alpha$ is a composition of $n$ and write $\alpha\vDash n$. We also write $\emptyset$ to denote the empty composition. A \textit{partition} {$\lambda=\lambda_1\cdots\lambda_{\ell(\lambda)}$} is a composition with parts satisfying $\lambda_1\ge\dots\ge\lambda_{\ell(\lambda)}$. If $\lambda$ is {a} partition satisfying $|\lambda|=n$, \tcr{we say that $\lambda$ is a partition of $n$ and write $\lambda\vdash n$.} The \textit{underlying partition} $\widetilde \alpha$ of a composition $\alpha$ is the partition obtained by sorting the parts of $\alpha$ in weakly decreasing order. When $\alpha$ is a composition with $\widetilde\alpha = n^{r_n}\cdots 1^{r_1}$, let $z_\alpha= 1^{r_1}r_1!\cdots n^{r_n}r_n!$.

Given two compositions $\alpha$ and $\beta$, their \textit{concatenation} is $\alpha\cdot\beta=\alpha_1\cdots\alpha_{\ell(\alpha)}\beta_1\cdots\beta_{\ell(\beta)}$. \tcr{The multiset of \textit{shuffles} $\alpha\shuffle\beta$ of two compositions $\alpha$ and $\beta$ consists of all compositions obtained by arranging the $\alpha_i$ and $\beta_i$ so that each $\alpha_i$ appears before $\alpha_{i+1}$ and each $\beta_i$ appears \VW{before} $\beta_{i+1}$.} The \textit{reversal} of a composition $\alpha$ is the composition $\alpha^r=\alpha_{\ell(\alpha)}\cdots\alpha_1$. Let $[n]=\{1,\dots,n\}$. If $\alpha\vDash n$, then we define $\set(\alpha)=\{\alpha_1,{\alpha_1+\alpha_2},\dots,\alpha_1+\dots+\alpha_{\ell(\alpha)-1}\}\subseteq [n-1]$. This induces a natural one-to-one correspondence between the compositions of $n$ and the subsets of $[n-1]$, and we let $\comp(S)$ denote the composition of $n$ associated \FA{with} a subset $S\subseteq[n-1]$. The \textit{complement} of $\alpha\vDash n$ is \VW{$\alpha^c=\comp(\set(\alpha)^c)$, where $\set(\alpha)^c=[n-1]\setminus\set(\alpha)$,} and the \textit{transpose} of $\alpha$ is $\alpha^t=(\alpha^r)^c=(\alpha^c)^r$. 

We say that $\alpha$ is a \textit{coarsening} of $\beta$  (or equivalently $\beta$ is a \textit{refinement} of $\alpha$), denoted \tcr{by} $\alpha \succcurlyeq\beta$, if we can obtain the parts of $\alpha$ in order by adding together adjacent parts of $\beta$ in order. \VW{In that case we will write $\beta^{(i)}$ to denote the composition consisting of the parts of $\beta$ \svw{that} sum to $\alpha_i$.} If $\alpha,\beta$ are compositions of the same size, \tcr{then} $\alpha\preccurlyeq\beta$ (or equivalently, $\alpha^c\succcurlyeq \beta^c$) if and only if $\text{set}(\beta)\subseteq\text{set}(\alpha)$. Finally, if $\alpha,\beta$ are compositions of $n$, their \textit{join} $\alpha\vee\beta$ is the unique composition of $n$ such that $\gamma\succcurlyeq \alpha$ and $\gamma\succcurlyeq\beta$ if and only if $\gamma\succcurlyeq\alpha\vee\beta$.

\tcr{\begin{example}\label{ex:comp}
If $\alpha = 3212 \vDash 8$ and $\beta \vDash 53$ then $\alpha \preccurlyeq \beta$ \VW{with $\alpha^{(2)}=12$} and $\alpha \vee \beta = \beta$, $\alpha \cdot \beta = 321253$.  Also $\widetilde\alpha = 3221 \vdash 8$, $\alpha ^r = 2123$, $\text{set} (\alpha) = \{3,5,6\}$, $\alpha ^c = \text{comp}(\{1,2,4,7\}) = 11231$, so $\alpha^t = 13211$.
\end{example}}

We now turn our attention to the Hopf algebra of symmetric functions $\Sym$, which may be realised as a subalgebra of $\mathbb Q[[x_1,x_2,\dots]]$. A formal power series $f\in \mathbb Q[[x_1,x_2,\dots]]$ is a \textit{symmetric function} if the degrees of the monomials appearing in $f$ are bounded and for every partition $\lambda$ all monomials $x_{i_1}^{\lambda_1}\tcr{\cdots} x_{i_{\ell(\lambda)}}^{\lambda_{\ell(\lambda)}}$ in $f$ with distinct indices $i_1,\dots,i_{\ell(\lambda)}$ have the same coefficient.

$\Sym=\Sym^0\oplus\Sym^1\oplus\cdots$ is graded, with $n$th graded {component} $\Sym^n$ spanned by the bases $\{m_\lambda\}_{\lambda\vdash n}$ and $\{p_\lambda\}_{\lambda\vdash n}$, \tcr{known as the \emph{$m$-basis} and \emph{$p$-basis}, respectively.} The \textit{monomial symmetric function} $m_\lambda$ \tcr{where $\lambda = \lambda _ 1 \cdots \lambda _{\ell(\lambda)}$} is
$$m_\lambda=\sum_{(i_1,\dots,i_{\ell(\lambda)})}x_{i_1}^{\lambda_1}\cdots x_{i_{\ell(\lambda)}}^{\lambda_{\ell(\lambda)}},$$
where the sum is over all $\ell(\lambda)$-tuples $(i_1,\dots,i_{\ell(\lambda)})$ of distinct indices that yield distinct monomials. \tcr{For example, $m_{211} = x_1 ^2 x_2x_3 + x_1 x_2^2 x_3 + \cdots$.} The \textit{power sum symmetric function} $p_\lambda$ \tcr{where $\lambda = \lambda _ 1 \cdots \lambda _{\ell(\lambda)}$} is
$$p_\lambda = \prod_{i=1}^{\ell(\lambda)} \sum_j x_j\tcr{^{\lambda _i}}.$$\tcr{For example, $p_{211} = (x_1 ^2  + x_2^2 + \cdots)(x_1+x_2 + \cdots)(x_1+x_2 + \cdots)$.}

$\Sym$ \tcr{is a} self-dual Hopf algebra, with one isomorphism $\Sym\cong\Sym^*$ given by the \textit{Hall inner product}, which is the bilinear form obtained by setting $\langle p_\lambda, p_\mu\rangle=z_\lambda\delta_{\lambda\mu}$, \tcr{and $\delta _{\lambda\mu} = 1$ if $\lambda = \mu$ and 0 otherwise, for example \cite[(4.7)]{MacD}}. There is a nice combinatorial {interpretation} for the coefficients of the expansion of $p_\lambda$ into the $m$-basis, which we give next.

\begin{proposition}\label{prop:ptomsym}\cite[Proposition 7.7.1]{EC2}
Let $\lambda$ be a partition of $n$. Then
$$p_\lambda=\sum_{\mu\vdash n} R_{\lambda\mu} m_\mu$$
where $R _{\lambda \mu}$ is the number of $\ell (\mu)\times \ell(\lambda)$ matrices $(m_{ij})$ whose nonzero entries are the parts of $\lambda$ such that
\begin{enumerate}
\item $\sum _{j=1} ^{\ell(\lambda)} m_{ij} = \mu _i$, that is, the entries of row $i$ sum to $\mu _i$,
\item the only nonzero entry in column $j$ is $\lambda _j$.
\end{enumerate}
\end{proposition}

\begin{example}\label{ex:ptomsym}
$$p_{211} = 2m _{211} + 2m_{31} + 2m_{22} + m_4$$from the following.
$$\begin{pmatrix}
2&&\\&1&\\&&1
\end{pmatrix}\quad
\begin{pmatrix}
2&&\\
&&1\\
&1&
\end{pmatrix}\quad
\begin{pmatrix}
2&1&\\&&1
\end{pmatrix}\quad
\begin{pmatrix}
2&&1\\&1&
\end{pmatrix}\quad
\begin{pmatrix}
2&&\\&1&1
\end{pmatrix}\quad
\begin{pmatrix}
&1&1\\2&&
\end{pmatrix}\quad
\begin{pmatrix}
2&1&1
\end{pmatrix}$$
\end{example}

In this paper, we are especially interested in the Hopf algebra of \textit{quasisymmetric functions}, which may also be realised as a subalgebra of $\mathbb Q[[x_1,x_2,\dots]]$. A formal power series $f\in\mathbb Q[[x_1,x_2,\dots]]$ is a \textit{quasisymmetric function} if the degrees of the monomials appearing in $f$ are bounded and for every composition $\alpha$ all monomials $x_{i_1}^{\alpha_1}\tcr{\cdots} x_{i_{\ell(\alpha)}}^{\alpha_{\ell(\alpha)}}$ in $f$ with distinct indices $i_1<\dots<i_{\ell(\alpha)}$ have the same coefficient. Note every symmetric function is a quasisymmetric function, and so let $\iota:\Sym\to\QSym$ denote the natural inclusion. Observe that this definition of $\QSym=\QSym(x_1,x_2,\dots)$ uses the infinite chain $\mathbb N$ to index our variables. In fact, one can use any infinite chain to define $\QSym$, as in {\cite[Definition 5.1.5]{GrinRein}.} 

Let $\{0,1\}\times\mathbb N$ be equipped with the lexicographic ordering, in which $(i,m)<(j,n)$ if and only if $i<j$ or $i=j$ with $m<n$. The coproduct $\Delta(f)$ of $f\in\QSym$ is computed by first considering $f(x_{(0,1)},x_{(0,2)},\dots,x_{(1,1)},x_{(1,2)},\dots)\in \QSym(x_{(0,1)},x_{(0,2)},\dots,x_{(1,1)},x_{(1,2)},\dots)$, where we consider quasisymmetric functions in variables indexed by the chain $\{0,1\}\times\mathbb N$. Then the coproduct $\Delta(f)$ is the image of $f(x_{(0,1)},x_{(0,2)},\dots,x_{(1,1)},x_{(1,2)},\dots)$ under the isomorphism $\QSym(x_{(0,1)},x_{(0,2)},\dots,x_{(1,1)},x_{(1,2)},\dots)\cong \QSym(x_1,x_2,\dots)\otimes \QSym(x_1,x_2,\dots)$ obtained by sending \tcr{$f_a(x_{(0,1)},x_{(0,2)},\dots)f_b(x_{(1,1)},x_{(1,2)},\dots)\mapsto f_a(x_1,x_2,\dots)\otimes f_b(x_1,x_2,\dots)$} and extending linearly {\cite[Section 5.1]{GrinRein}}.

$\QSym=\QSym^0\oplus\QSym^1\oplus\cdots$ is graded, with $n$th graded {component} $\QSym^n$ spanned by the bases $\{M_\alpha\}_{\alpha\vDash n}$ and $\{F_\alpha\}_{\alpha\vDash n}$, \tcr{known as the \emph{monomial} and \emph{fundamental basis}, respectively.} The \textit{monomial quasisymmetric function} $M_\alpha$ \tcr{where $\alpha = \alpha _1 \cdots \alpha _{\ell(\alpha)}$} is
$$M_\alpha = \sum_{i_1<\dots<i_{\ell(\alpha)}}x_{i_1}^{\alpha_1}\cdots x_{i_{\ell(\alpha)}}^{\alpha_{\ell(\alpha)}}.$$\tcr{For example, $M_{112} = x_1x_2x_3^2 + x_1x_2x_4^2+\cdots $.}
Note that the monomial quasisymmetric functions naturally refine the monomial symmetric functions via
$$m_\lambda=\sum_{\widetilde\alpha=\lambda}M_\alpha.$$
The \textit{fundamental quasisymmetric function} $F_\alpha$ \tcr{where $\alpha = \alpha _1 \cdots \alpha _{\ell(\alpha)}$} is
$$F_\alpha = \sum_{\substack{i_1\le\dots\le i_{\ell(\alpha)}\\ i_j<i_{j+1}\text{ if }j\in\set(\alpha)}} x_{i_1}\cdots x_{i_{\ell(\alpha)}}.$$\tcr{For example, $F_{112} = x_1x_2x_3^2 + x_1x_2x_3x_4+\cdots $.}
The monomial and fundamental bases are related by the equation
$$M_{\alpha}=\sum_{\beta\preccurlyeq\alpha} (-1)^{\ell(\beta)-\ell(\alpha)}F_\beta.$$

Three important involutory automorphisms $\QSym\to\QSym$ are defined on the fundamental basis by \tcr{$$\psi(F_\alpha)=F_{\alpha^c},\qquad \rho(F_\alpha)=F_{\alpha^r},\qquad \omega(F_\alpha)=F_{\alpha^t}$$}\tcr{for example, \cite{Gelfandetal}, although we follow the notation of \cite[Section 3.6]{LMvW}.} The restriction of $\omega:\QSym\to\QSym$ to $\Sym$ is an involutory automorphism $\omega:\Sym\to\Sym$ \tcr{ satisfying $\omega(p_\lambda)= (-1)^{n-\ell(\lambda)} p_\lambda$ for $\lambda \vdash n$ \cite[(2.13)]{MacD}.}

The Hopf algebra of noncommutative symmetric functions $\NSym$ can be defined as the {graded Hopf dual} of $\QSym$ {(see \cite[Section 3.4.2]{LMvW}).} It is spanned by the basis $\{\mathbf h_\alpha\}_{\alpha\vDash n\ge 0}$, which is dual to the monomial basis $\{M_{\alpha}\}_{\alpha\vDash n\ge 0}$ of $\QSym$. That is, each $\langle M_\alpha,\mathbf h_\beta\rangle =\delta_{\alpha\beta}$, where \tcr{$\delta_{\alpha\beta} = 1$ if $\alpha = \beta$ and 0 otherwise, and} $\langle\cdot,\cdot\rangle$ is the inner product between $\QSym$ and its dual. There is a surjective morphism $\chi:\NSym\to \Sym$, defined implicitly by letting $\langle f,\chi(g)\rangle = \langle \iota(f),g\rangle$ for all $f\in\Sym$ and all $g\in\NSym$. The involutory antiautomorphism $\omega:\NSym\to\NSym$ {can be} defined implicitly by letting $\langle \omega( f),g\rangle = \langle f,\omega( g)\rangle$ for all $f\in\QSym$ and $g\in \NSym$ {\cite[Section 3.6]{LMvW}}.

Given a set $S$ and a nonnegative integer $k$, write ${S\choose k}$ \VW{to} denote the set of all $k$-element subsets of $S$. Finally, we will also use the shorthand $[b_i]f$ to denote the coefficient of {$b_i$ when $f$ is} expanded into a basis $\{b_i\}_{i\in I}$.

\section{Weighted $P$-partitions}\label{sec:P-part}

In this section we will introduce the generating function of $P$-partitions associated \FA{with} a weighted labelled poset $(P,\gamma,w)$ and investigate its basic properties. \tcr{Our definition generalises that of \cite{AlexSulz} where the poset is naturally labelled, and is itself further generalised to weighted double posets in \cite{Grinberg}.} Given a poset $P$, a \textit{weight function} is a map $w: P\to\mathbb N$. A triple $(P,\gamma, w)$ where $(P,\gamma)$ is a labelled poset and $w$ is a weight function is a \textit{weighted labelled poset}, \tcr{with \emph{total weight} $w(P) = \sum _{u\in P} w(u)$}.

\begin{definition}\label{def:KPgw}
Given a weighted labelled poset $(P,\gamma,w)$, we define the generating function
$$K_{(P,\gamma,w)}=\sum_f \prod_{u\in P} x_{f(u)}^{w(u)},$$
where the sum is over all $(P,\gamma)$-partitions $f:P\to\mathbb N$.
\end{definition}

\tcr{It is straightforward to see from its definitions that $K_{(P,\gamma,w)}$ is a quasisymmetric function. Also note that if $w(u)=1$ for all $u\in P$, then $K_{(P,\gamma,w)}$ is the usual generating function for $P$-partitions, also called the weight enumerator.}
\tcr{\begin{example}\label{ex:weightedp} Continuing Example~\ref{ex:ppartition} with $\gamma(p)=1$ and $\gamma(q)=2$, let us now attach weights $w(p)=5$ and $w(q)=3$. Then
$$K_{(P,\gamma,w)}= x_1^5x_2^3+x_1^3x_2^5+\cdots +x_1^8+x_2^8+\cdots = (x_1^5+x_2^5+\cdots)(x_1^3+x_2^3+\cdots)= K_{(p,\gamma,w)}K_{(q,\gamma,w)}.$$
\end{example}}

\tcr{Therefore we will determine the} product and coproduct rules for these generating functions, as well as the effects of the quasisymmetric involutions $\psi$, $\rho$ and $\omega$ \tcr{next. Concrete examples are in Section~\ref{sec:power}, where we apply these results to develop \VW{the basis of} combinatorial power sums, apart from the following, for which we have the previous example.}

\begin{proposition}\label{prop:Kprod}
Let $(P\dotcup Q,\gamma,w)$ be a weighted labelled poset. Then
$$K_{(P,\gamma,w)} K_{(Q,\gamma,w)}=K_{(P\dotcup Q,\gamma,w)}.$$
\end{proposition}
\begin{proof}
From Definition~\ref{def:P-part}, it is clear \tcr{that} since elements of $P$ are incomparable to elements of $Q$ in $P\dotcup Q$ that $\mathscr O(P\dotcup Q,\gamma)$ consists exactly of the maps $f:P\dotcup Q\to \mathbb N$ such that the restrictions $f:P\to \mathbb N$ and $f:Q\to\mathbb N$ are a $(P,\gamma)$- and a $(Q,\gamma)$-partition, respectively.

Therefore,
$$K_{(P,\gamma,w)}K_{(Q,\gamma,w)}=\sum_{f_1,f_2}\left(\prod_{u\in P}x_{f_1(u)}^{w(u)} \prod_{v\in Q}x_{f_2(v)}^{w(v)}\right) =K_{(P\dotcup Q,\gamma,w)} ,$$
as desired, where the sum is over all $(P,\gamma)$-partitions $f_1:P\to\mathbb N$ and all $(Q,\gamma)$-partitions $f_2:Q\to\mathbb N$.
\end{proof}

\begin{proposition} \label{prop:Kcop}
Let $(P,\gamma,w)$ be a weighted labelled poset. Then
$$\Delta(K_{(P,\gamma,w)}) = \sum_{I} K_{(I,\gamma,w)}\otimes K_{(P\setminus I,\gamma,w)},$$
where the sum is over all {lower sets} $I$ of $P$.
\end{proposition}
\begin{proof}
By the definition of $K_{(P,\gamma,w)}$ and the coproduct on $\QSym$,
$$\Delta(K_{(P,\gamma,w)}) = \sum_{f_1,f_2}\left( \prod_{f_1(u)=0} x_{f_2(u)}^{w(u)}\otimes \prod_{f_1(v)=1} x_{f_2(v)}^{w(v)}\right),$$
where the sum is over all pairs of maps $f_1:P\to \{0,1\}$ and $f_2:P\to \mathbb N$ satisfying, for all $p<q$ in $P$, \tcr{(1)} that $(f_1(p),f_2(p))\le (f_1(q),f_2(q))$ in $\{0,1\}\times \mathbb N$ with the lexicographic ordering and \tcr{(2)} $(f_1(p),f_2(p))=(f_1(q),f_2(q))$ only if $\gamma(p)<\gamma(q)$.

Given such a pair $(f_1,f_2)$, the preimage $f_1^{-1}(0)$ must be {a lower set} of $P$, since $(f_1,f_2):P\to \{0,1\}\times\mathbb N$ is order-preserving and $\{0\}\times\mathbb N$ is {a lower set} of $\{0,1\}\times\mathbb N$. Given {a lower set} $I$ of $P$, a pair $(f_1,f_2)$ with $f_1^{-1}(0)=I$ satisfies the two conditions if and only if the restrictions $f_2: I\to \mathbb N$ and $f_2:P\setminus I\to \mathbb N$ are an $(I,\gamma)$- and a $(P\setminus I,\gamma)$-partition, respectively. It follows that
$$\Delta(K_{(P,\gamma,w)}) = \sum_I \sum_{f_2} \left(\prod_{u\in I} x_{f_2(u)}^{w(u)}\otimes \prod_{v\in P\setminus I}x_{f_2(v)}^{w(v)}\right) = \sum_I K_{(I,\gamma,w)}\otimes K_{(P\setminus I,\gamma, w)},$$
where the inner sum is over all maps $f_2:P\to\mathbb N$ such that the restrictions to $I$ and $P\setminus I$ are an $(I,\gamma)$- and a $(P\setminus I,\gamma)$-partition, respectively.
\end{proof}

We briefly direct our attention to weighted labelled chains and the expansions of the generating functions of their $P$-partitions into the monomial and fundamental bases.

\begin{definition}
Let $(s,\gamma,w)$ be a weighted labelled chain, where $s$ is given by $u_1<\dots<u_\ell$. Then $\alpha(s,w)$ is the composition \tcr{$w(u_1)\cdots w(u_\ell)$} and $\delta(s,\gamma,w)\succcurlyeq \alpha(s,w)$ is the composition $\comp(\{\sum_{j=1}^k w(u_j): \gamma(u_k)>\gamma(u_{k+1})\})$.
\end{definition}

\tcr{\begin{example}\label{ex:funnycomps}
Let us consider the chain $s= u_1<u_2<u_3$ with  labels $\gamma (u_1)=3, \gamma(u_2)=1, \gamma(u_3)=2$, and $w(u_1)=7, w(u_2)=5, w(u_3)=4$. Then $\alpha(s,w)=754$ and $\delta(s,\gamma,w)=\text{comp}(\{7\})=79$.
\end{example}}

\begin{lemma}\label{lem:chaintoM}
Let $(s,\gamma,w)$ be a weighted labelled chain. Then
$$K_{(s,\gamma, w)}=\sum_{\alpha(s,w)\preccurlyeq\beta\preccurlyeq\delta(s,\gamma,w)}M_\beta.$$
\end{lemma}
\begin{proof}
Let $s$ be given by $u_1<\dots<u_\ell$. An order-preserving map $f:s\to\mathbb N$ is an $(s,\gamma)$-partition if and only if $\{u_j:\gamma(u_j)>\gamma(u_{j+1})\}\subseteq\{u_j:f(u_j)<f(u_{j+1})\}$.

Therefore,
$$K_{(s,\gamma,w)}=\sum_{\set(\delta(s,\gamma,w))\subseteq S\subseteq \set(\alpha(s,w))}\sum_{(i_1,\dots,i_{w(s)})}x_{i_1}\cdots x_{i_{w(s)}},$$
where the inner sum is over all $w(s)$-tuples $(i_1,\dots,i_{w(s)})$ such that $i_1\le\dots\le i_{w(s)}$ with $i_j<i_{j+1}$ if and only if {$j\in S$}. So
\begin{align*}
    K_{(s,\gamma,w)}&=\sum_{\set(\delta(s,\gamma,w)\subseteq S\subseteq \set(\alpha(s,w))} M_{\comp (S)}=\sum_{\alpha(s,w)\preccurlyeq\beta\preccurlyeq\delta(s,\gamma,w)}M_\beta.
\end{align*}
\end{proof}

\begin{lemma}\label{lem:chaintoF}
Let $(s,\gamma,w)$ be a weighted labelled chain. Then
$$K_{(s,\gamma,w)}=\sum_{ \delta(s,\gamma^*,w)^c\preccurlyeq \beta\preccurlyeq \delta(s,\gamma,w)} (-1)^{\ell(\beta)-\ell(\delta(s,\gamma,w))} F_\beta.$$
\end{lemma}
\begin{proof}
Note we may write $$K_{(s,\gamma,w)}=\sum_{(i_1,\dots, i_{w(s)})} x_{i_1}\cdots x_{i_{w(s)}},$$ where the sum is over all $w(s)$-tuples $(i_1,\dots,i_{w(s)})$ of positive integers satisfying $i_1\le \dots \le i_{w(s)}$ such that $i_k=i_{k+1}$ whenever $k\not\in \set(\alpha(s,w))= \set(\delta(s,\gamma,w))\cup \set(\delta(s,\gamma^*,w))$ and $i_k<i_{k+1}$ whenever $k\in \set(\delta(s,\gamma,w))$.

By the principle of inclusion-exclusion,
\begin{align*}
    K_{(s,\gamma,w)} &= \sum_{\set(\delta(s,\gamma,w))\subseteq S\subseteq\set(\delta(s,\gamma^*,w))^c}(-1)^{|S|-|\set(\delta(s,\gamma,w))|}F_{\comp(S)}\\&=\sum_{ \delta(s,\gamma^*,w)^c\preccurlyeq \beta\preccurlyeq \delta(s,\gamma,w)} (-1)^{\ell(\beta)-\ell(\delta(s,\gamma,w))} F_\beta.
\end{align*}\end{proof}

The expansion into fundamental \tcr{quasisymmetric functions} in Lemma~\ref{lem:chaintoF} will be useful in computing the effects of the involutions $\psi$, $\rho$ and $\omega$ on $K_{(P,\gamma,w)}$.

\begin{lemma}
\label{lem:chaininv}
Let $(s,\gamma, w)$ be a weighted labelled chain. Then
\begin{enumerate}
    \item $\psi (K_{(s,\gamma,w)}) =(-1)^{w(s)-|s|} K_{(s,\gamma^*,w)}$,
    \item $\rho (K_{(s,\gamma,w)}) = K_{(s^*,\gamma^*,w)}$,
    \item $\omega (K_{(s,\gamma,w)}) =(-1)^{w(s)-|s|} K_{(s^*,\gamma,w)}$.
\end{enumerate}
\end{lemma}
\begin{proof}
Applying $\psi$ to Lemma~\ref{lem:chaintoF},
\begin{align*}
    \psi(K_{(s,\gamma,w)}) &= \sum_{\delta(s,\gamma,w)^c\preccurlyeq \beta \preccurlyeq \delta(s,\gamma^*,w)} (-1)^{\ell(\delta(s,\gamma,w)^c)-\ell(\beta)}F_\beta\\
    &=(-1)^{\ell(\delta(s,\gamma,w)^c)-\ell(\delta(s,\gamma^*,w))} K_{(s,\gamma^*,w)},
\end{align*}
where the second line is also by Lemma~\ref{lem:chaintoF}. Then (1) follows since $\ell(\delta(s,\gamma,w)^c)-\ell(\delta(s,\gamma^*,w))=w(s)-|s|$. The proofs of (2) and (3) are similar, noting by the definitions that $\delta(s,\gamma,w)^r=\delta(s^*,\gamma^*,w)$ and $\delta(s,\gamma^*,w)^r=\delta(s^*,\gamma,w)$.
\end{proof}

\begin{proposition}
\label{prop:posetinv}
Let $(P,\gamma, w)$ be a weighted labelled poset. Then
\begin{enumerate}
    \item $\psi (K_{(P,\gamma,w)}) =(-1)^{w(P)-|P|} K_{(P,\gamma^*,w)}$,
    \item $\rho (K_{(P,\gamma,w)}) = K_{(P^*,\gamma^*,w)}$,
    \item $\omega (K_{(P,\gamma,w)}) =(-1)^{w(P)-|P|} K_{(P^*,\gamma,w)}$.
\end{enumerate}
\end{proposition}
\begin{proof}
By Lemma~\ref{lem:P-part},
$$K_{(P,\gamma,w)}=\sum_s K_{(s,\gamma,w)},$$
where the sum is over all linear extensions $s$ of $P$. Therefore, by applying Lemma~\ref{lem:chaininv},
$$\psi(K_{(P,\gamma,w)})=(-1)^{w(P)-|P|} \sum_s K_{(s,\gamma^*,w)}= (-1)^{w(P)-|P|} K_{(P,\gamma^*,w)},$$
with the right equality again by Lemma~\ref{lem:P-part}, proving (1). The proofs of (2) and (3) are similar, noting that the duals $s^*$ of the linear extensions $s$ of $P$ are exactly the linear extensions of the dual poset $P^*$.
\end{proof}

\begin{remark}
Grinberg in \cite{Grinberg} studied generating functions of $\mathbf E$-partitions associated \FA{with} a weighted double poset, which generalise the functions $K_{(P,\gamma,w)}$. Since the antipode on $\QSym$ differs from $\omega:\QSym\to\QSym$ by a sign on each graded {component}, Proposition~\ref{prop:posetinv} (3) may be deduced from \tcr{\cite[Theorem 8]{Grinberg},} which computes the antipode of $K_{(P,\gamma,w)}$ in a special case. Because we study specifically $P$-partitions associated \FA{with} a labelled poset, we obtain a simpler proof by employing the fundamental lemma of $P$-partitions.
\end{remark}

\section{Quasisymmetric and noncommutative power sums}\label{sec:qsymnsym}
In this section we investigate what properties should be satisfied by quasisymmetric and noncommutative analogues of power sums. Motivated by the relationship $\langle p_\lambda,p_\mu\rangle = z_\lambda \delta_{\lambda\mu}$ for symmetric power sums, we will study the consequences of a similar relationship between quasisymmetric and {noncommutative} power sums. Given bases $\{P_\alpha\}_{\alpha\vDash n\ge 0}$ and $\{\mathbf p_\alpha\}_{\vDash n\ge 0}$ of $\QSym$ and $\NSym$, respectively, related by $\langle P_\alpha,\mathbf p_\beta\rangle = z_\alpha \delta_{\alpha\beta}$, we should expect that $\{P_\alpha\}_{\alpha\vDash n\ge 0}$ is a good quasisymmetric analogue of \tcr{the $p$-basis} exactly whenever $\{\mathbf p_\alpha\}_{\alpha\vDash n\ge 0}$ is a good {noncommutative} analogue of \tcr{the $p$-basis.}

The following proposition describes the relationship between the properties satisfied by a quasisymmetric power sum basis and the properties satisfied by a {noncommutative} power sum basis when the two bases are scaled duals of one another.

\begin{proposition}\label{prop:dual}
Let $\{P_{\alpha}\}_{\alpha\vDash n\ge 0}$ and $\{\mathbf p_\alpha\}_{\alpha\vDash n\ge 0}$ be bases for $\QSym$ and $\NSym$, respectively, related by $\langle P_\alpha, \mathbf p_\beta\rangle = z_\alpha\delta_{\alpha\beta}$. Then
\begin{enumerate}
    \item $\sum_{\widetilde{\alpha}=\lambda}P_\alpha = p_\lambda$ for all partitions $\lambda$ if and only if $\chi(\mathbf p_\alpha) = p_{\widetilde{\alpha}}$ for all compositions $\alpha$,
    \item $P_\alpha P_\beta = \frac{z_\alpha z_\beta}{z_{\alpha\cdot\beta}}\sum_{\gamma\in\alpha\shuffle\beta}P_\gamma$ for all compositions $\alpha,\beta$ if and only if \tcr{$\Delta(\mathbf p_\gamma)=\sum_{{\gamma\in\alpha\shuffle\beta}}\mathbf p_\alpha\otimes\mathbf p_\beta$ for all compositions $\gamma$,}
    \item $\Delta(P_\alpha)=\sum_{\alpha=\beta\cdot\gamma}\frac{z_\alpha}{z_\beta z_\gamma}P_\beta\otimes P_\gamma$ for all compositions $\alpha$ if and only if \tcr{$\mathbf p_{\beta}\mathbf p_\gamma=\mathbf p_{\beta\cdot\gamma}$ for all compositions $\beta,\gamma$,}
    \item $\omega(P_\alpha)=(-1)^{|\alpha|-\ell(\alpha)}P_{\alpha^r}$ for all compositions $\alpha$ if and only if $\omega(\mathbf p_\alpha)=(-1)^{|\alpha|-\ell(\alpha)}\mathbf p_{\alpha^r}$ for all compositions $\alpha$,
    \item $[M_\beta]P_\alpha\ge 0$ for all compositions $\alpha,\beta$ if and only if \tcr{$[\mathbf p_\alpha]\mathbf h_\beta\ge 0$} for all compositions $\alpha,\beta$,
    \item $[M_\beta]P_\alpha=0$ for all compositions $\alpha\not\preccurlyeq\beta$ if and only if \tcr{$[\mathbf p_\alpha]\mathbf h_\beta = 0$} for all compositions \tcr{$\alpha\not\preccurlyeq\beta$.}
\end{enumerate}
\end{proposition}
\begin{proof}
To prove (1), note that \tcr{$$[P_{\alpha}]p_\lambda = \frac{1}{z_\alpha}\langle \iota(p_\lambda), \mathbf p_\alpha\rangle = \frac{1}{z_\alpha}\langle p_\lambda,\chi(\mathbf p_\alpha)\rangle=\frac{z_\lambda}{z_\alpha}[p_\lambda]\chi(\mathbf p_\alpha),$$}and both conditions \tcr{in (1)} are equivalent to having this quantity \tcr{above} equal $\delta_{\widetilde{\alpha}\lambda}$ for all compositions $\alpha$ and all partitions $\lambda$. For (2), we apply the equality $$[P_\gamma]P_\alpha P_\beta = \frac{1}{z_\gamma}\langle P_\alpha P_\beta, \mathbf p_\gamma\rangle =\frac{1}{z_\gamma}\langle P_\alpha\otimes P_\beta, \Delta(\mathbf p_\gamma)\rangle = \frac{z_\alpha z_\beta}{z_\gamma} \tcr{[\mathbf p_\alpha \otimes \mathbf p_\beta]} \Delta(\mathbf p_\gamma),$$ and note for $\gamma\in \alpha\shuffle\beta$ that $z_\gamma=z_{\alpha\cdot\beta}$ since $\gamma$ and $\alpha\cdot\beta$ have the same parts. Similarly, for (3) we apply the equality \tcr{$$[P_\beta\otimes P_\gamma] \Delta(P_\alpha) = \frac{1}{z_\beta z_\gamma}\langle \Delta(P_\alpha), \mathbf p_\beta\otimes \mathbf p_\gamma\rangle = \frac{1}{z_\beta z_\gamma}\langle P_\alpha,\mathbf p_\beta\mathbf p_\gamma\rangle = \frac{z_\alpha}{z_\beta z_\gamma}[\mathbf p_\alpha] \mathbf p_\beta\mathbf p_\gamma .$$} We deduce (4) from the equality \tcr{$$[P_\beta]\omega(P_\alpha) = \frac{1}{z_\beta}\langle \omega(P_\alpha), \mathbf p_\beta\rangle = \frac{1}{z_\beta}\langle P_\alpha, \omega(\mathbf p_\beta)\rangle =\frac{z_\alpha}{z_\beta} [\mathbf p_\alpha] \omega(\mathbf p_\beta)$$}and use that $z_{\alpha^r}=z_\alpha$. Finally, (5) and (6) follow from the fact that $$\VW{[M_\beta]P_\alpha = \langle P_\alpha, \mathbf h_\beta\rangle =z_\alpha [\mathbf p_\alpha] \mathbf h_\beta.}$$
\end{proof}

\begin{remark}\label{rem:types}
The bases of type 1 quasisymmetric power sums $\{\Psi_\alpha\}_{\alpha\vDash n\ge 0}$ and type 2 quasisymmetric power sums $\{\Phi_\alpha\}_{\alpha\vDash n\ge 0}$, introduced by Ballantine, Daugherty, Hicks, Mason and Niese \cite{BDHMN} as the scaled duals of the noncommutative power sums of the first and second kind \cite{Gelfandetal}, \VW{were shown in the same paper to satisfy all of properties (1), (2), (4), (5) and (6). Property (3) for the scaled duals of the type 1 and 2 quasisymmetric power sums appears in \cite{Gelfandetal}.}
\end{remark}

Properties (1) to (6) generalise properties of the \tcr{$p$-basis}  to quasisymmetric and {noncommutative symmetric} functions. The first property allows the quasisymmetric functions $\{P_\alpha\}_{\alpha\vDash n\ge 0}$ to naturally refine the power sum symmetric functions.

The \tcr{$p$-basis} $\{p_\lambda\}_{\lambda\vdash n\ge 0}$ is the basis of $\Sym$ obtained by taking all finite products of the $\{p_n\}_{n\ge 1}$, where each $p_n$ is primitive, meaning $\Delta(p_n)=1\otimes p_n + p_n\otimes 1$, and is an eigenvector of $\omega:\Sym\to\Sym$ with eigenvalue $(-1)^{n-1}$. Property (3) makes $\{\mathbf p_\alpha\}_{\alpha\vDash n\ge 0}$ the basis of $\NSym$ obtained by taking all finite products of the $\{\mathbf p_n\}_{n\ge 1}$. Assuming \VW{property (3)}, property (2) is {equivalent to} each $\mathbf p_n$ being primitive, since then each $\Delta(\mathbf p_n)=\mathbf p_\emptyset\otimes\mathbf p_n + \mathbf p_n\otimes \mathbf p_\emptyset = 1\otimes\mathbf p_n+\mathbf p_n\otimes 1$, and the coproduct on $\mathbf p_\alpha$ can then be found by using the property that the coproduct of a Hopf algebra is an algebra morphism. \VW{Moreover, assuming} property (3), property (4) is equivalent to each $\mathbf p_n$ being an eigenvector of $\omega:\NSym\to\NSym$ with eigenvalue $(-1)^{n-1}$, since $\omega:\NSym\to\NSym$ is an antiautomorphism.

Property (5) generalises the \tcr{positivity in the $m$-basis of the power sum symmetric functions to the quasisymmetric monomial basis} and power sums, and property (6) makes the indexing of the quasisymmetric power sums $\{P_\alpha\}_{\alpha\vDash n\ge 0}$ by compositions natural. So we expect a good quasisymmetric analogue of power sum \tcr{symmetric} functions to satisfy these six properties. In the next section, we will introduce the basis $\{\p_\alpha\}_{\alpha\vDash n\ge 0}$ of combinatorial power sums for $\QSym$, for which we will show these properties hold.

\section{The basis of combinatorial power sums}\label{sec:power}

We will define \VW{two new bases} of $\QSym$ naturally refining the power sum symmetric functions. \VW{The same bases were independently discovered by Lazzeroni \cite{Lazz}, who gave a different definition.} Our definition is motivated by the observation that the power sum symmetric function $p_\lambda$ is the generating function of the $P$-partitions associated \FA{with} a certain weighted labelled poset.

First, let $\mathbb N_{\mathbb N}$ be the set $\{a_k:a,k\in\mathbb N\}$, equipped with the order
$$1_1<1_2<\cdots<2_1<2_2<\cdots<3_1<3_2<\cdots.$$ The elements of $\mathbb N_{\mathbb N}$ are \textit{labelled positive integers}.
The dual chain $\mathbb N_{\mathbb N}^*$ is given by
$$1_1>1_2>\cdots>2_1>2_2>\cdots>3_1>3_2>\cdots.$$

Given a partition $\lambda=n^{r_n}\cdots 1^{r_1}$, let $P^{\lambda}$ be the poset with no relations and underlying set $\{n_{r_n},\dots,n_{1},\dots, 1_{r_1},\dots,1_1\}$. Now let 
\tcr{$$\gamma:P^{\lambda}\to \mathbb N_{\mathbb N}^*$$}be the natural inclusion and let $w:P^\lambda\to\mathbb N$ be the weight function given by $w(a_k)=a$. \tcr{We will see the reason for choosing the dual chain in Remark~\ref{rem:dualchain}.} Then, since $P^\lambda$ has no relations, $\mathscr O(P^{\lambda},\gamma)$ simply consists of all maps $f:P^{\lambda}\to\mathbb N$. Therefore,
$$K_{(P^{\lambda},\gamma,w)}=\sum_{f}x_{f(n_{r_n})}^n\cdots x_{f(n_1)}^n\cdots x_{f(1_{r_1})}\cdots x_{f(1_1)}=p_\lambda,$$
where the sum is \tcr{over} all maps $f:P^{\lambda}\to\mathbb N$. Applying Lemma~\ref{lem:P-part}, the fundamental lemma of $P$-partitions, to $K_{(P^{\lambda},\gamma,w)}$ then gives a decomposition into quasisymmetric functions, which we will use to define our new basis of quasisymmetric power sums.

Before defining our basis, we will introduce some more notation. A \textit{labelled composition} $\alpha_*$ is a finite ordered list of distinct labelled positive integers. We will write $(\alpha_*)$ to denote the composition obtained by removing subscripts. For $i\in\mathbb N$, we write $\ind_i(\alpha_*)$ to denote the set of indices $k$ for all $i_k$ appearing in $\alpha_*$. For example, when $\alpha_*=1_31_12_3$, we have $(\alpha_*)=112$ and $\ind_1(\alpha_*)=\{1,3\}$. Labelled compositions provide succinct notation to describe finite chains on labelled positive integers.

We will write $K_{\alpha_*}$ for a labelled composition $\alpha_*$ to denote $K_{(s,\gamma,w)}$, where $\alpha_*$ lists the elements of  {a} chain $s$ in ascending order, $\gamma$ is the natural inclusion into $\mathbb N_{\mathbb N}^*$ and the weight function $w$ is given by $w(a_k)=a$. Note $\alpha(s,w)=(\alpha_*)$.
 We will also define $K^r_{\alpha_*} = K_{(s,\gamma^*,w)}$. We are now ready to define the elements of our new basis.

\begin{definition}\label{def:cpow}
The \textit{combinatorial power sum} quasisymmetric function indexed by a composition $\alpha$ with $\widetilde\alpha=n^{r_n}\cdots 1^{r_1}$ is
$$\p_\alpha= \sum_{\alpha(s,w)=\alpha}K_{(s,\gamma,w)}=\sum_{\substack{(\alpha_*)=\alpha\\ \ind_i(\alpha_*)=[r_i],  {i\in\mathbb N}}} K_{\alpha_*},$$
where the first sum is over all linear extensions $s$ of $P^{\widetilde\alpha}$ satisfying $\alpha(s,w)=\alpha$. 
\end{definition}

\begin{example}
When $\alpha=112$, the linear extensions $s$ of $P^{211}$ satisfying $\alpha(s,w)=112$ are given by $1_1<1_2<2_1$ and $1_2<1_1<2_1$, and so
$$\p_{112} = K_{1_11_22_1} + K_{1_2 1_12_1}= 2M_{112}+M_{22},$$
since by Lemma~\ref{lem:chaintoM}, $K_{1_11_22_1}=M_{112}$ and $K_{1_21_12_1}=M_{112}+M_{22}$. \tcr{Similarly, we can compute the following.
$$\p_{4}= \FA{K_{4_1}=}M_4\qquad \p _{31}=\FA{K_{3_11_1}=} M_4 + M_{31}\qquad \p _{13}=\FA{K_{1_13_1}=} M_{13}$$ $$\p _{22}=\FA{K_{2_12_2}+K_{2_22_1}=} M_4 + 2M_{22}
\qquad \p _{211}=\FA{K_{2_11_11_2}+K_{2_11_21_1}=} M_4+ 2M_{31} + M_{22} + 2 M_{211}$$$$ \p _{121}=\FA{K_{1_12_11_2}+K_{1_22_11_1}=}  2M _{13} + 2M_{121}\qquad \p _{112}=\FA{K_{1_11_22_1}+K_{1_21_12_1}=}  M_{22} + 2M_{112}$$
$$\p_{1111}=\FA{\sum K_{1_i1_j1_k1_l}=}  M_4+ 4M_{31} + 4M_{13} + 6M_{22} + 12M_{211}  +12 M_{121} + 12 M_{112} + 24 M_{1111}$$} \FA{where the sum runs over all choices of distinct $i,j,k,l\in[4]$.}
\end{example}
There is a second family of quasisymmetric functions \tcr{that} can be defined similarly.

\begin{definition}\label{def:rcpow}
The \textit{reverse combinatorial power sum} quasisymmetric function indexed by a composition $\alpha$ with $\widetilde\alpha=n^{r_n}\cdots 1^{r_1}$ is
$$\p^r_\alpha= \sum_{\alpha(s,w)=\alpha}K_{(s,\gamma^*,w)}=\sum_{\substack{(\alpha_*)=\alpha\\ \ind_i(\alpha_*)=[r_i],  {i\in\mathbb N}}} K^r_{\alpha_*},$$
where the first sum is over all linear extensions $s$ of $P^{\widetilde\alpha}$ satisfying $\alpha(s,w)=\alpha$. 
\end{definition}

By Lemma~\ref{lem:chaintoM}, it follows that each $\p_\alpha$ expands with nonnegative integer coefficients in the monomial basis of $\QSym$, with $[M_\alpha]\p_\alpha>0$. By an upper-triangularity argument against the monomial basis, $\{\p_\alpha\}_{\alpha\vDash n\ge 0}$ is a basis for $\QSym$. Note that the basis $\{\p_\alpha\}_{\alpha\vDash n\ge 0}$ is distinct from the bases $\{\Psi_{\alpha}\}_{\alpha\vDash n\ge 0}$ and $\{\Phi_\alpha\}_{\alpha\vDash n\ge 0}$ of quasisymmetric power sums introduced in \cite{BDHMN}, since although they expand into the monomial basis with nonnegative coefficients, those coefficients are not always \tcr{integers.}

Since
$$p_\lambda = K_{(P^\lambda,\gamma,w)} = \sum_{s}K_{(s,\gamma,w)},$$
where the sum is over all linear extensions $s$ of $P^\lambda$ and since the parts of each $\alpha(s,w)$ are exactly the weights of the elements of $P^\lambda$, we have the following theorem immediately by the definition of the $\{\p_\alpha\}_{\alpha\vDash n\ge 0}$, \tcr{and the above discussion.}

\begin{theorem}\label{the:symptoqsymp}
Let $\lambda$ be a partition. Then
$$p_\lambda = \sum_{\widetilde{\alpha}=\lambda} \p_\alpha.$$\tcr{Moreover, the set $\{ \p _\alpha \} _{\alpha \vDash n \geq 0}$ is a basis for $\QSym$.}
\end{theorem}
\begin{example}
$$p_{211}= \p _{211}+\p _{121} + \p_{112}$$
{Note that the expansion into monomial quasisymmetric functions of the right side agrees with the expansion into monomial symmetric functions of Example~\ref{ex:ptomsym}.}
\end{example}
One useful observation is the following. If $\alpha_*,\beta_*$ are two labelled compositions satisfying $(\alpha_*)=(\beta_*)$  {such} that the indices associated \FA{with} {each} unlabelled positive integer $i$ are in the same relative order in $\alpha_*$ and $\beta_*$, {then $K_{\alpha_*}=K_{\beta_*}$}. For example, $K_{1_21_12_1}=K_{1_31_12_3}$, since $(1_21_12_1)=(1_31_12_3)=112$, and the indices of $1_21_1$ are in the same relative order as those of $1_31_1$. This is because there is an order-preserving bijection between the associated weighted labelled chains that also preserves weights and relative labelling. Therefore, for any composition $\alpha$ with $\widetilde\alpha=n^{r_n}\cdots 1^{r_1}$, we also have
$$\p_\alpha = \sum_{\substack{(\alpha_*)=\alpha\\ \ind_i(\alpha_*)=[r_i],  {i\in\mathbb N}}}K_{\alpha_*} = \sum_{\substack{(\alpha_*)=\alpha\\ \ind_i(\alpha_*)=R_i,  {i\in\mathbb N}}}K_{\alpha_*}$$ for any choice of the $R_i\subset \mathbb N$ of size $r_i$.

We next prove a shuffle product rule for the basis of combinatorial power sums, \tcr{where we extend the notion of shuffles to two labelled compositions $\alpha_*$ and $\beta_*$ with no parts in common in the natural way.}

\begin{theorem}\label{the:pprod}
Let $\alpha,\beta$ be compositions. Then
$$\p_\alpha \p_\beta = \frac{z_\alpha z_\beta}{z_{\alpha\cdot\beta}}\sum_{\gamma\in\alpha\shuffle\beta}\p_\gamma.$$
\end{theorem}
\tcr{\begin{example}\label{ex:prodp}
$$\p _{12}\p _1 = \frac{2}{4}\left(\p _{121} + \p _{112} + \p _{112} \right) = \frac{1}{2} \p _{121} + \p _{112}$$
\end{example}}
\begin{proof}
Let $\widetilde\alpha=n^{r_n}\cdots 1^{r_1}$ and $\widetilde\beta=m^{t_m}\cdots 1^{t_1}$. Note as in \cite[\tcr{Notation 3.15}]{BDHMN}, $\frac{z_{\alpha\cdot\beta}}{z_\alpha z_\beta} = \prod_{i} {r_i+t_i\choose r_i}$. Thus
\begin{align*}
\frac{z_{\alpha\cdot\beta}}{z_\alpha z_\beta} \p_\alpha\p_\beta&= \sum_{ {(R_1, R_2, \ldots )\in \prod_{i} {[r_i+t_i]\choose r_i}}} \left(\sum_{\substack{(\alpha_*)=\alpha\\ \ind_i(\alpha_*)=R_i,  {i\in\mathbb N}}}K_{\alpha_*}\right)\left(\sum_{\substack{(\beta_*)=\beta\\ \ind_i(\beta_*)=[r_i+t_i]\setminus R_i,  {i\in\mathbb N}}}K_{\beta_*}\right).
\end{align*}

When $\alpha_*$ and $\beta_*$ have no parts in common, the product $K_{\alpha_*}K_{\beta_*}$ is the generating function of the $P$-partitions of the disjoint union of the weighted labelled chains associated \FA{with} $\alpha_*$ and $\beta_*$, by Proposition~\ref{prop:Kprod}. The labelling on the disjoint union is the natural inclusion into $\mathbb N_{\mathbb N}^*$. By Lemma~\ref{lem:P-part}, the fundamental lemma of $P$-partitions, this generating function may be written as a sum over linear extensions, where the linear extensions of the disjoint union are exactly the chains indexed by the shuffles of $\alpha_*$ and $\beta_*$.

Hence,
\begin{align*}
\frac{z_{\alpha\cdot\beta}}{z_\alpha z_\beta} \p_\alpha\p_\beta&= \sum_{ {(R_1, R_2, \ldots )\in \prod_{i} {[r_i+t_i]\choose r_i}}}\sum_{\substack{(\alpha_*)=\alpha,\,(\beta_*)=\beta\\ \ind_i(\alpha_*)=R_i,\,\ind_i(\beta_*) = [r_i+t_i]\setminus R_i,  {i\in\mathbb N}}}\sum_{\gamma_* \in \alpha_*\shuffle\beta_*} K_{\gamma_*}\\
&=\sum_{\gamma\in\alpha\shuffle\beta} \sum_{\substack{(\gamma_*)=\gamma\\ \ind_i(\gamma_*)=[r_i+t_i],  {i\in\mathbb N}}}K_{\gamma_*} = \sum_{\gamma\in\alpha\shuffle\beta} \p_\gamma,
\end{align*}
proving the desired product rule.
\end{proof}

Our next result is a deconcatenation coproduct rule for our basis. When $\alpha_*,\beta_*$ are two labelled compositions with no parts in common, we let $\alpha_*\cdot\beta_*$ denote their concatenation.

\begin{theorem}\label{the:pcop}
Let $\alpha$ be a composition. Then
$$\Delta(\p_\alpha)=\sum_{\alpha=\beta\cdot\gamma}\frac{z_\alpha}{z_\beta z_\gamma}\p_\beta\otimes\p_\gamma.$$
\end{theorem}
\tcr{\begin{example}\label{ex:coprodp}
$$\Delta (\p _{121}) = 1\otimes \p _{121} + 2 \p_1 \otimes \p_{21} + 2 \p _{12} \otimes \p_1 + \p _{121} \otimes 1$$
\end{example}}
\begin{proof}
Let $\widetilde\alpha=n^{r_n}\cdots 1^{r_1}$. Then
\begin{align*}
    \Delta(\p_\alpha) = \sum_{\substack{(\alpha_*)=\alpha\\ \ind_i(\alpha_*)=[r_i],  {i\in\mathbb N}}} \Delta(K_{\alpha_*}) = \sum_{\substack{(\alpha_*)=\alpha\\ \ind_i(\alpha_*)=[r_i],  {i\in\mathbb N}}} \sum_{\alpha_*=\beta_*\cdot\gamma_*}K_{\beta_*}\otimes K_{\gamma_*},
\end{align*}
where the second equality is by Proposition~\ref{prop:Kcop}. Therefore,
$$\Delta(\p_\alpha)=\sum_{\alpha=\beta\cdot\gamma}\sum_{ {(T_1, T_2, \ldots )\in \prod_{i} {[r_i]\choose t_i}}} \sum_{\substack{(\beta_*)=\beta,\,(\gamma_*)=\gamma\\ \ind_i(\beta_*)=T_i,\,\ind_i(\gamma_*)=[r_i]\setminus T_i,  {i\in\mathbb N}}} K_{\beta_*}\otimes K_{\gamma_*} = \sum_{\alpha=\beta\cdot\gamma} \frac{z_\alpha}{z_{\beta}z_\gamma} \p_\beta\otimes \p_\gamma,$$
where $t_i$ denotes the number of $i$'s appearing in $\beta$. The second equality is because $\frac{z_\alpha}{z_\beta z_\gamma} = \prod_i{r_i\choose t_i}$ when $\alpha=\beta\cdot\gamma$, as noted in \cite[\tcr{Notation 3.15}]{BDHMN}.
\end{proof}

The effects of the involutions $\psi$, $\rho$ and $\omega$ on our basis can be deduced from our work studying these involutions on weighted $P$-partition generating functions.

\begin{theorem}\label{the:pinv}
Let $\alpha$ be a composition of $n$. Then
\begin{enumerate}
    \item 
    $\psi(\p_\alpha)= (-1)^{n-\ell(\alpha)} \p_\alpha^r$,
    \item 
    $\rho(\p_\alpha)=\p^r_{\alpha^r}$,
    \item 
    $\omega(\p_\alpha)=(-1)^{n-\ell(\alpha)}\p_{\alpha^r}.$
\end{enumerate}
\end{theorem}
\tcr{\begin{example}\label{ex:mapps}
$$\psi (\p _{112}) = - \p_{112}^r \qquad \rho (\p _{112}) = \p _{211} ^r \qquad \omega (\p _{112}) = - \p _{211}$$
\end{example}}
\begin{proof}
We compute by Lemma~\ref{lem:chaininv} that
$$\psi(\p_\alpha)= \sum_{\alpha(s,w)=\alpha} \psi(K_{(s,\gamma,w)})=(-1)^{n-\ell(\alpha)}\sum_{\alpha(s,w)=\alpha}K_{{(s,\gamma^*,w)}}=(-1)^{n-\ell(\alpha)} \p_\alpha^r,$$
where the sum is over all linear extensions $s$ of $P^{\widetilde\alpha}$ satisfying $\alpha(s,w)=\alpha$, proving (1). The proofs of (2) and (3) are similar, noting that $\alpha(s^*,w)=\alpha(s,w)^r$ for any finite chain $s$ and any weight function $w:s\to\mathbb N$.
\end{proof}
{We illustrate Theorem~\ref{the:pinv} with the following diagram when $\alpha$ is a composition of $n$.}
\vspace{2pt}
\begin{center}
\vspace{5mm}
\begin{tikzcd}
\p_{\alpha} \arrow[dd, "\rho", mapsto] \arrow[rrrr, "\psi", mapsto]\arrow[rrrrdd, "\omega", mapsto] & && & (-1)^{n-\ell(\alpha)}\p_{\alpha}^r\arrow[dd, "\rho", mapsto] \\ \\
\p^r_{\alpha^r}\arrow[rrrr, "\psi", mapsto]& &&& (-1)^{n-\ell(\alpha)}\p_{\alpha^r}
\end{tikzcd}
\vspace{5mm}
\end{center}

The last two theorems of this section describe combinatorially the coefficients of our basis when expanded into the \tcr{monomial and fundamental} bases of $\QSym$.

\begin{theorem}\label{the:ptomqsym}
Let $\alpha$ be a composition of $n$. Then
$$\p_\alpha = \sum_{\beta\vDash n}\mathcal R_{\alpha\beta}M_\beta,$$
where $\mathcal{R} _{\alpha \beta}$ is the number of $\ell (\beta)\times \ell(\alpha)$ matrices $(m_{ij})$ whose nonzero entries are the parts of $\alpha$ such that
\begin{enumerate}
\item $\sum _{j=1} ^{\ell(\alpha)} m_{ij} = \beta _i$, that is, the entries of row $i$ sum to $\beta _i$,
\item the only nonzero entry in column $j$ is $\widetilde{\alpha} _j$,
\item the nonzero entries give $\alpha$ when read left to right and top to bottom.\end{enumerate}
\end{theorem}

\begin{example}\label{ex:ptomqsym}
$$\mathcal{p}_{121} = 2M _{121} + 2M_{13}$$from the following.
$$\begin{pmatrix}
&1&\\
2&&\\
&&1
\end{pmatrix}\quad
\begin{pmatrix}
&&1\\2&&\\&1&
\end{pmatrix}\quad
\begin{pmatrix}
&1&\\2&&1
\end{pmatrix}\quad
\begin{pmatrix}
&&1\\2&1&
\end{pmatrix}
$$
\end{example}

\begin{proof}
We will prove this by exhibiting a one-to-one correspondence between each matrix $(m_{ij})$ and each term $M_\beta$ that we obtain when we compute $\p_\alpha$.

By Definition~\ref{def:cpow} and Lemma~\ref{lem:chaintoM},
$$\p_\alpha = \sum_{\alpha(s,w)=\alpha}\sum_{\alpha(s,w)\preccurlyeq\beta\preccurlyeq\delta(s,\gamma,w)}M_\beta.$$
Let $\widetilde\alpha=n^{r_n}\cdots 1^{r_1}$. The first $r_n$ columns of $(m_{ij})$ will correspond to $n_{r_n}, \ldots , n_{1}$ and each column will contain one $n$ representing that weight, the next $r_{n-1}$ columns of $(m_{ij})$ will correspond to $(n-1)_{r_{n-1}}, \ldots , (n-1)_1$ and each column will contain one $n-1$ representing that weight, $\ldots\ $, the last $r_1$ columns of $(m_{ij})$ will correspond to $1_{r_1}, \ldots , 1_1$ and each column will contain one $1$ representing that weight.
(The other entries in each column are zero.) 

To construct our correspondence, note that  {(by property (3))} the reading word of $(m_{ij})$ gives a linear extension $s$ of $P^{\widetilde\alpha}$ satisfying $\alpha(s,w)=\alpha$. Moreover, the ordering of our columns guarantees that two elements of $P^{\widetilde\alpha}$ are in the same row of $(m_{ij})$ only if the labelling $\gamma$ preserves their relative order in $s$. Placing the $k$th $a$ from the right in row $i$ then corresponds to weight $a_k$ contributing towards part $\beta_i$ in $M_\beta$.
\end{proof}

\begin{remark}
Observe that Theorem~\ref{the:ptomqsym} naturally refines Proposition~\ref{prop:ptomsym} and we can deduce the proposition as a corollary via the following. Note if $\lambda$ is a partition of $n$ and
$$p_\lambda = \sum _{\mu\vdash n} R _{\lambda \mu} m _\mu = \sum _{\mu\vdash n} R _{\lambda \mu} \left(\sum _{\tilde{\beta} = \mu} M_\beta \right)$$
for some coefficients $R_{\lambda\mu}$, then to enumerate how many $m_\mu$ we have we need only choose a representative $\beta$ such that $\tilde{\beta} = \mu$ to represent $\sum _{\tilde{\beta} = \mu} M_\alpha$. So choose ${{\beta} = \mu}$. We can compute the coefficient of $M_\mu$ in the expansion of $p_\lambda$ into the monomial basis of $\QSym$ by applying Theorems~\ref{the:symptoqsymp} and \ref{the:ptomqsym}. Note that since there is no restriction given by a linear extension, we no longer need the third condition of Theorem~\ref{the:ptomqsym}. Thus $R _{\lambda \mu}$ is the number of $\ell (\mu)\times \ell(\lambda)$ matrices $(m_{ij})$ whose nonzero entries are the parts of $\lambda$ such that
\begin{enumerate}
\item $\sum _{j=1} ^{\ell(\lambda)} m_{ij} = \mu _i$, that is, the entries of row $i$ sum to $\mu _i$,
\item the only nonzero entry in column $j$ is $\lambda _j$.
\end{enumerate}
\end{remark}

\VW{
We find an explicit formula for the coefficients $\mathcal R_{\alpha\beta}$ in the following proposition, which shows that all nonzero coefficients in the expansion of the combinatorial power sums in the \svw{monomial} basis may be written as a product of multinomial coefficients.
\begin{proposition} Let $\alpha\preccurlyeq\beta$ be compositions of $n$. Let $s_a$ denote the number of parts of size $a$ in $\beta$ and $t_a^{(i)}$ the number of parts of size $a$ in $\alpha^{(i)}$ for $1\le a\le n$ and $1\le i \le \ell(\beta)$. Then
$$\mathcal R_{\alpha\beta} = \left(\prod_{a=1}^n{s_a\choose t_a^{(1)},\dots, t_a^{(\ell(\beta))}}\right)\left(\prod_{i=1}^{\ell(\beta)} \delta_{\widetilde{\alpha^{(i)}} \alpha^{(i)}}\right).$$
\end{proposition}}

\begin{proof}
\VW{Any matrix counted by $\mathcal R_{\alpha\beta}$ must have its nonzero entries in row $i$ form, in order, the composition $\alpha^{(i)}$. Since the only nonzero entry in column $j$ is $\widetilde\alpha_j$, the composition formed by the nonzero entries in row $i$ is a partition. Hence $\mathcal R_{\alpha\beta}=0$ whenever some $\alpha^{(i)}$ is not a partition.}

\VW{Otherwise, if each $\alpha^{(i)}$ is a partition, then each matrix counted by $\mathcal R_{\alpha\beta}$ is obtained by making a choice for each $1\le a\le n$ of how to partition \svw{$s_a$} $a$'s among the rows of the matrix $(m_{ij})$ so that each row $i$ has $t_a^{(i)}$ $a$'s. In this case we then have
$$\mathcal R_{\alpha\beta} = \prod_{a=1}^n{\svw{s_a}\choose t_a^{(1)},\dots, t_a^{(\ell(\beta))}}.$$}
\end{proof}

\begin{theorem}\label{the:ptofqsym}
Let $\alpha$ be a composition of $n$. Then
$$\p_\alpha=\sum_{\beta\vDash n}(-1)^{\ell(\beta)-\ell(\alpha\vee\beta)}\mathcal Q_{\alpha(\alpha\vee\beta)} F_\beta,$$
where  {$\mathcal{Q} _{\alpha \xi}$} is the number of $ {\ell (\xi)}\times \ell(\alpha)$ matrices $(m_{ij})$ whose nonzero entries are the parts of $\alpha$ such that

\begin{enumerate}
\item $\sum _{j=1} ^{\ell(\alpha)} m_{ij} =  {\xi _i}$, that is, the entries of row $i$ sum to $ {\xi _i}$,
\item the only nonzero entry in column $j$ is $\widetilde{\alpha} _j$,
\item the nonzero entries give $\alpha$ when read left to right and top to bottom,
\item the rightmost nonzero entry of row $i$ is right of the leftmost nonzero entry of row $i+1$.
\end{enumerate}
\end{theorem}

\begin{example}
$$\p_{121}=-2F_{112}+2F_{13}$$
from the following.
$$\begin{pmatrix}
&1&\\2&&1
\end{pmatrix}\quad
\begin{pmatrix}
&&1\\2&1&
\end{pmatrix}
$$
\end{example}
\begin{proof}
We have by Theorem~\ref{the:ptomqsym} and the transition \tcr{matrix} between the monomial and \tcr{fundamental} quasisymmetric functions that
\begin{align*}
    \p_\alpha &= \sum_{\gamma\vDash n}\mathcal R_{\alpha\gamma}M_\gamma = \sum_{\gamma\succcurlyeq \alpha} \mathcal R_{\alpha\gamma}M_\gamma\\
    &= \sum_{\gamma\succcurlyeq\alpha} \mathcal R_{\alpha\gamma}\left(\sum_{\beta\preccurlyeq\gamma} (-1)^{\ell(\beta)-\ell(\gamma)}F_\beta\right)\\
    &=\sum_{\beta\vDash n} \left(\sum_{\gamma\succcurlyeq \alpha\vee\beta}(-1)^{\ell(\beta)-\ell(\gamma)} \mathcal R_{\alpha\gamma}\right) F_\beta.
\end{align*}
Therefore, it remains to show the identity
$$\mathcal Q_{\alpha(\alpha\vee\beta)} = \sum_{\gamma\succcurlyeq \alpha\vee\beta} (-1)^{\ell(\alpha\vee\beta)-\ell(\gamma)}\mathcal R_{\alpha\gamma}.$$

Note for $\gamma\succcurlyeq \alpha\vee\beta$ that the matrices counted by $\mathcal R_{\alpha\gamma}$ are naturally in bijection with the matrices counted by $\mathcal R_{\alpha(\alpha\vee\beta)}$ satisfying the additional property that the rightmost nonzero entry of row $i$ is  {left} of the leftmost nonzero entry of row $i+1$ whenever $(\alpha\vee\beta)_i$ and $(\alpha\vee\beta)_{i+1}$ correspond to the same part of $\gamma$ in the coarsening $\gamma\succcurlyeq \alpha\vee\beta$. The desired identity then follows from the definition of $\mathcal Q_{\alpha(\alpha\vee\beta)}$ and applying the principle of inclusion-exclusion.
\end{proof}

\begin{remark}\label{rem:dualchain}
The analogous theorems for \tcr{the reverse} combinatorial power sums $\{\p_\alpha^r\}_{\alpha\vDash n\ge 0}$ are proved by the same arguments. In particular, they \tcr{form a basis for $\QSym$ and} naturally refine the power sum symmetric functions via 
\tcr{$$p_\lambda= \sum_{\widetilde\alpha=\lambda}\p^r_\alpha,$$} multiply via the product rule 
\tcr{$$\p_\alpha^r \p_\beta^r = \frac{z_\alpha z_\beta}{z_{\alpha\cdot\beta}}\sum_{\gamma\in\alpha\shuffle\beta}\p_\gamma^r$$} and comultiply via the coproduct rule 
\tcr{$$\Delta(\p_\alpha^r)=\sum_{\alpha=\beta\cdot\gamma}\frac{z_\alpha}{z_\beta z_\gamma}\p_\beta^r\otimes\p_\gamma^r.$$} Moreover, when $\alpha\vDash n$ we have 
\tcr{$$\psi(\p_\alpha^r)=(-1)^{n-\ell(\alpha)}\p_\alpha,\qquad \rho(\p_\alpha^r)=\p_{\alpha^r},\qquad \omega(\p_{\alpha}^r)=(-1)^{n-\ell(\alpha)}\p_{\alpha^r}^r$$} and $\p_\alpha^r$ expands into the monomial and fundamental bases via 
\tcr{$$\p_\alpha^r=\sum_{\beta\vDash n}\mathcal R_{\alpha^r\beta^r}M_\beta = \sum_{\beta\vDash n}(-1)^{\ell(\beta)-\ell(\alpha\vee\beta)}\mathcal Q_{\alpha^r(\alpha^r\vee\beta^r)}F_\beta.$$}

One advantage of the basis of combinatorial power sums over the basis of reverse combinatorial power sums is that, e.g. by Theorem~\ref{the:ptomqsym}, $[M_{|\alpha|}]\p_{\alpha}=\delta_{\widetilde\alpha\alpha}$, although we similarly have $[M_{|\alpha|}]\p^r_\alpha = \delta_{\widetilde\alpha \alpha^r}$. That is, if $\alpha$ is a composition of $n$, the coefficient of $M_n$ in the monomial expansion of $\p_\alpha$ is $1$ if $\alpha$ is a partition, and $0$ otherwise.
\end{remark}
\section{Acknowledgements}\label{sec:ack} The authors would like to thank  {the referees,} Per Alexandersson and Bruce Sagan  for helpful suggestions and comments.

\bibliographystyle{plain}

\end{document}